\documentclass[11pt]{amsart}
\usepackage[margin=1in]{geometry}
\usepackage{amsmath}
\usepackage{amsfonts,amsmath}
\usepackage{paralist}
\usepackage[colorlinks=true]{hyperref}
\hypersetup{urlcolor=blue, citecolor=red}
\numberwithin{equation}{section}
\newtheorem{theorem}{Theorem}[section]

\newtheorem{lemma}[theorem]{Lemma}

\theoremstyle{definition}
\newtheorem{definition}[theorem]{Definition}

\newcommand*{\avint}{\mathop{\ooalign{$\int$\cr$-$}}}
\newcommand{\ot}{\Omega_T }
\newcommand{\pot}{\partial_p\Omega_T }
\newcommand{\po}{\partial\Omega}
\newcommand{\mdiv}{\textup{div}}

\newcommand{\nvp}{\nabla\varphi}
\newcommand{\vp}{\varphi}
\newcommand{\pt}{\partial_t}
\newcommand{\ep}{\varepsilon}
\newcommand{\su}{\sigma(u)}

\newcommand{\io}{\int_{\Omega}}
\newcommand{\ioT}{\int_{\Omega_{T}}}
\newcommand{\aiy}{\avint_{B_r(y)}  }

\title[ a nonlinear elliptic-parabolic system with oscillating electrical conductivity  
] %Use the shortened version of the full title
{ Regularity results for a nonlinear elliptic-parabolic system with oscillating coefficients
}

\author[Xiangsheng Xu]{}

\email{xxu@math.msstate.edu}
\begin{document}
	\maketitle
	
	% Enter the first author's name and address:
	\centerline{\scshape Xiangsheng Xu}
	\medskip
	{\footnotesize
		% please put the address of the first author
		\centerline{Department of Mathematics \& Statistics}
		\centerline{Mississippi State University}
		\centerline{ Mississippi State, MS 39762, USA}
	} % Do not forget to end the {\footnotesize by the sign }

	\bigskip

	%The abstract of your paper	% This result is a little bit surprising in view of the fact that classical regularity theory for elliptic equations like the first one in our system requires 
		\begin{abstract}
		In this paper we study the initial boundary value problem for the system $\mbox{div}(\sigma(u)\nabla\varphi)=0,\ \ u_t-\Delta u=\sigma(u)|\nabla\varphi|^2$. This problem is known as the thermistor problem which models the electrical heating of conductors. Our assumptions on $\sigma(u)$ leave open the possibility that  $\liminf_{u\rightarrow\infty}\sigma(u)=0$, while $\limsup_{u\rightarrow\infty}\sigma(u)$ is large. This means that $\sigma(u)$ can oscillate wildly between $0$ and a large positive number as $u\rightarrow \infty$. Thus our degeneracy is fundamentally different from the one that is present in porous medium type of equations. We obtain a weak solution $(u, \varphi)$ with $|\nabla \varphi|, |\nabla u|\in L^\infty$ by first establishing a uniform upper bound for $e^{\varepsilon u}$ for some small $\varepsilon$. This leads to an inequality in $\nabla\varphi$, from whence follows the regularity result. This approach enables us to avoid first proving the H\"{o}lder continuity of $\varphi$ in the space variables, which would have required that the elliptic coefficient $\sigma(u)$ be an $A_2$ weight.  As it is known, the latter implies that $\ln\sigma(u)$ is ``nearly bounded''.
		
	\end{abstract}

	%\mathbb{R}^2$
	%The title of your section 1
	\section{Introduction}
Let $\Omega$ be a bounded domain in $\mathbb{R}^N$ with sufficiently smooth boundary $\po$ and $T$ any positive number. We consider the initial boundary value problem
\begin{eqnarray}
u_t-\Delta u&=&\su|\nvp|^2\ \ \mbox{in $\ot$},\label{me1}\\
\mdiv(\su\nvp)&=& 0\ \ \mbox{in $\ot$},\label{me2}\\
u&=&u_0\ \ \ \mbox{on $\partial_p\ot$},\label{me3}\\
\varphi&=&\varphi_0\ \ \mbox{on $\Sigma_T$,}\label{me4}
\end{eqnarray}
where 
\begin{eqnarray}
\ot&=&\Omega\times(0,T),\\
\Sigma_T&=&\po\times(0,T),\ \mbox{the lateral boundary of $\ot$, and}\\ \pot&=&\Sigma_T\cup \Omega\times\{0\},\ \mbox{the parabolic boundary of $\ot$}.
\end{eqnarray}
We are interested in the regularity properties of weak solutions  when the elliptic coefficient $\sigma(u)$ in the second equation may become oscillatory as $u\rightarrow \infty$. 
%Before we introduce our results, we define our notion of a weak solution.
To be precise, we establish the following
\begin{theorem}[Main Theorem]\label{thm11}
	%Let $\Omega$ be a bounded domain in $\mathbb{R}^2$ with $C^3$ boundary $\partial\Omega$ and \eqref{uop} be satisfied. 
	Assume:
	\begin{enumerate}
		\item[\textup{(H1)}] the function $\sigma$ is continuously differentiable on the interval $[0,\infty)$ with
		% and bounded above by a linear function and below by an exponential function, i.e.,
		%	there exists a positive number $d$ such that
		\begin{eqnarray}\label{cons}
		c_0e^{-\beta s}&\leq& \sigma(s)\leq c_1\ \ \mbox{on $[0,\infty)$ for some $c_0,c_1, \beta\in (0,\infty)$ and}\\
		|\sigma^\prime(s)|&\leq& c_2e^{\gamma s}\ \ \mbox{ on $[0,\infty)$ for some $c_2, \gamma\in (0,\infty)$};
		%\frac{|\sigma^\prime(s)|}{\sigma(s)}\leq d\ \ \mbox{for $s\in [0,\infty)$};
		\end{eqnarray}
		\item[\textup{(H2)}] $ u_0, \varphi_0\in C\left([0,T]; C^1(\overline{\Omega})\right)$ with $u_0\mid_{\partial_p\ot}\geq 0$, $\partial_tu_0\in L^2(\ot)$, and $\Delta\varphi_0\in L^\infty(0,T;L^s(\Omega)) $ for each $s>1$;
		\item[\textup{(H3)}]  $\po$ is $C^{1,1}$.
	\end{enumerate}
	Then there is a weak solution $(u,\varphi)$ to \eqref{me1}-\eqref{me4} with $u\geq 0$ and
	\begin{equation}\label{nunv1}
	\nabla u,\nabla\varphi\in L^\infty(\ot).
	\end{equation} 
\end{theorem}
%We study the system on $\mathbb{R}

The notion of a weak solution is defined as follows:
\begin{definition}
	We say that $(u,\vp)$ is a weak solution to \eqref{me1}-\eqref{me4} if
	\begin{enumerate}
		\item[\textup{(D1)}] $u, \vp\in L^2(0,T; W^{1,2}(\Omega))$;
		\item[\textup{(D2)}] $u=u_0, \vp=\vp_0$ on $\Sigma_T$ in the sense of the trace theorem and 
		\begin{eqnarray}
		-\ioT u\xi_tdxdt+\ioT\nabla u\nabla\xi dxdt&=&\ioT\su|\nvp|^2dxdt+\io u_0(x,0)\xi(x,0)dx,\\
		\ioT\su\nvp\nabla\xi dxdt&=&0
		\end{eqnarray}
		for each smooth function $\xi$ with $\xi\mid_{\Sigma_T}=0$ and $\xi(x, T)=0$.
	\end{enumerate}
\end{definition}
We quickly offer another perspective on the initial condition for $u$. The weak maximum principle asserts that
\begin{equation}\label{wmp}
\|\vp\|_{\infty,\ot}\leq \|\vp_0\|_{\infty,\ot}.
\end{equation} We can easily derive from \eqref{me2} that
\begin{equation}
\ioT\su|\nvp|^2\xi dxdt=-\ioT\su\vp\nvp\nabla\xi dxdt \ \ \mbox{for each $\xi$  with $\xi\mid_{\Sigma_T}=0$.}
\end{equation}
This together with \eqref{me1} implies that $u_t\in L^2(0,T; W^{-1,2}(\Omega))$. Thus we can conclude that $u\in C([0,T], L^2(\Omega))$. The initial condition $u(x,0)=u_0(x,0)$ can also be understood to hold in this space.

Physically, problem \eqref{me1}-\eqref{me4} may be proposed as a model for the electrical heating of a conductor, the so-called thermistor problem.
%A mathematical description of this is given by  It states that 
In this case $u$ is the temperature and $\varphi$ the electrical potential  of the conductor.
%, which is represented by a bounded domain $\Omega$ in $\mathbb{R}^N$, are governed by the following system of partial differential equations
% and $T$ is any positive number. 
The heat source is the Joule heating $\su\nvp\cdot\nvp$, where  $\su$ is the temperature-dependent electrical conductivity.  We have taken the thermal conductivity to be $1$. 
%Precise assumptions on $\su$ will be made later. 
%Let  be  with boundary . For any $T>0$ set
%The system is coupled with the initial boundary conditions
There is a large body of literature devoted to the study of \eqref{me1}-\eqref{me4} under various assumptions on $\sigma(s)$ and the boundary conditions and various extensions. For the mathematical analysis of the associated stationary problem, we would like to mention \cite{C,CP,HRS}.
Modeling and numerical simulations were investigated in \cite{ALZ,KBO,WK}. For optimal control issues, we refer the reader to \cite{HK} and the references therein. Also see \cite{ALM} and its references for obstacle thermistor problems. %See, e.g., \cite{X6} and the references therein.
% In almost all of the existing work on the non-stationary case like ours it is assumed that $\sigma(s)$ is bounded above. In fact, if $\sigma(s)$ is unbounded, it is often associated with the finite-time blow up of $u$ \cite{AC}. This point becomes very clear in view of the fact 
Of course, there are many more papers that we have failed to mention, and it is simply beyond the scope of this paper to give a comprehensive review of the current research in this area.  

A very important issue about the time-dependent problem is: How does one prevent the thermal run-away from occurring? The blow-up of solutions was studied in \cite{AC}. In applications, blow-up of solutions are not welcome in general. Thus we will focus our attention on the boundedness of $u$. If $\sigma$ is also bounded away from $0$ below, then \eqref{me2} becomes uniformly elliptic and one has
%\eqref{phc} is simply a consequence of the classical regularity theory for linear elliptic equations \cite{Y}.  we recall a result in \cite{Y} which asserts that $u$ is H\"{o}lder continuous in $\overline{\ot}$ if
\begin{equation}\label{phc}
\varphi\in L^\infty(0, T; C^\alpha(\overline{\Omega}))\ \ \mbox{for some $\alpha\in (0,1)$.} 
\end{equation}
This combined with a result in \cite{Y}  asserts that $u$ is H\"{o}lder continuous in $\overline{\ot}$, from whence follows
%This, in turn, implies 
that for each $p>1$ there is a positive number $c$ depending on the continuity of $\su$ and $C^{1,1}$ boundary such that
\begin{equation}\label{r21}
\|\nvp\|_{p,\ot}\leq c\|\nvp_0\|_{p,\ot}
\end{equation}
(\cite{R}, p.82).
%By differentiating \eqref{me1} with respect to $x_i$, $i=1,\cdot, N$, respectively, 
We can easily infer from the proof of Lemma \ref{nunv} below that \eqref{r21} implies $|\nabla u|\in L^\infty(\ot)$. Now write \eqref{me2} in the form
\begin{equation}\label{efvp}
\Delta\varphi=-\frac{\sigma^\prime(u)}{\su}\nabla u\cdot\nvp.
\end{equation}
This puts us in a position to apply the classical Calder\'{o}n-Zygmund estimate \cite{CFL}. Upon doing so, we establish \eqref{nunv1}.

%To gain some perspectives on condition \eqref{osc},
Under (H1), the problem immediately becomes very delicate because we have to leave open the possibility that $u$ is not bounded above. The reason is simple: The term on the right-hand of \eqref{me1} is only an $L^1$ function from the usual energy estimates. Consequently, \eqref{me2} could become degenerate and a priori estimates are difficult to obtain. In fact, even an $L^p$, $p\geq 1$, estimate for $\nabla\varphi$ is unlikely unless additional assumptions on $\sigma$ are made \cite{X1,X2}. Furthermore,  assumption (H1) allows the possibility that
\begin{equation}\label{osc}
\limsup_{s\rightarrow\infty}\sigma(s)>0\ \ \mbox{and}\ \ \liminf_{s\rightarrow\infty}\sigma(s)=0
\end{equation}
hold simultaneously. This means that the function $\sigma(s)$ can oscillate wildly between $0$ and a positive number as $s\rightarrow\infty$. We can easily come up with an example of such functions. Say, 
%take
%For example, we can take $\sigma(s)$ to be a function of the form
\begin{equation*}
\sigma(s)=c_3(1+\sin e^{\gamma s})+c_0e^{-\beta s}, \ \ c_3>0.
\end{equation*}
%To obtain a local version of \eqref{phc},
% that Recall that 
%we see from 
By virtue of the classical regularity theory \cite{HKM} for degenerate and/or singular elliptic equations of the type \eqref{me2}, $\su$ must be an $A_2$-weight for \eqref{phc} to hold. We say that $\su$ is an $A_2$-weight if there is a positive number $c$ such that
\begin{equation}\label{atw}
\aiy\su dx\aiy\frac{1}{\su}dx\leq c\ \ \mbox{for all $y\in \Omega, r>0$ with $B_r(y)\subset\Omega$,}
\end{equation} 
where $B_r(y)$ denotes the open ball centered at $y$ with radius $r$.
A theorem in (\cite{ST}, p.141) asserts that
a function $f$ is an $A_2$ weight if and only if  $\ln f$ belongs to BMO. The latter implies that over any ball, the average oscillation of $\ln f$ must be bounded. In the situation considered here, to obtain \eqref{atw} we have to assume  $\su= e^{-c u}$ for some $c>0$ according to a result in \cite{X4}. 
%that $\su$ is an $A_2$ weight  when . $ is roughly a constant multiple of the function $
In general, our main theorem seems to lie outside the scope of \cite{HKM}. This is the main motivation for our study.
%there is 
%Then a bootstrap argument employing the results in \cite{GT,Y} appropriately can lead to the conclusion of our main theorem.
%It is not difficult to infer from  and Lemma \ref{nunv} that Theorem \ref{thm11} can be derived from 
%the H\"{o}lder continuity of $\varphi$ in the space variables (uniformly in the time variable).
%boundedness of $u$. 
In a series of three papers (\cite{X1}-\cite{X3}), the author obtained the boundedness of $u$ under the assumptions that the given function  $\sigma(s)$  has the properties:
\begin{enumerate}
	\item[(C1)] $\sigma(s)$ is continuous, positive, and bounded above;
	\item[(C2)] $\lim_{s\rightarrow\infty}\sigma(s)=0$; and
	\item[(C3)] $\lim_{\tau\rightarrow 0^+}\frac{\sigma(s+\tau)}{\sigma(s)}=1$ uniformly on $[0,\infty)$.
\end{enumerate}
%	It is fairly easy to see that (C3) follows from \eqref{cons}. Indeed, consider the function
%	\begin{equation}
%	g(\tau)=\frac{\sigma(s+\tau)}{\sigma(s)}.
%	\end{equation}
%	Then we have
%	\begin{equation}
%	g^\prime(\tau)=\frac{\sigma^\prime(s+\tau)}{\sigma(s)}=g(\tau)\frac{\sigma^\prime(s+\tau)}{\sigma(s+\tau)}.
%	\end{equation}
%	This yields
%	\begin{equation}
%	-d\leq \left(\ln g(\tau)\right)^\prime\leq d.
%	\end{equation}
%	Integrate and take note of the fact that $g(0)=1$ to obtain
%	\begin{equation}
%	e^{-d\tau}\leq g(\tau)\leq e^{d\tau}.
%	\end{equation}
%	This gives (C3).
In particular, condition (C2) is essential to the argument there. We have managed to remove this condition here, thereby allowing oscillation in $\sigma$.
A result in \cite{X3} asserts that (C3) implies that $\su$ is bounded below by an exponential function. Thus we have also weaken (C3) substantially.
% in comparison to the previous work \cite{X3}. 
The trade-off for us is that we have to assume that  $\sigma$ is continuously differentiable. 
%We must point out that the physical relevance of \eqref{osc} is not clear. 

%That is, we allow 
%Even though the physical relevance of \eqref{osc} is not clear,
%we have not been able to obtain a physical meaning for 
%\eqref{osc}
%Even though the physical relevance of this condition is yet to come, mathematically it is 
%This is
%condition \eqref{osc} 
%it certainly poses a new challenge for the regularity of solutions \cite{X6}.
%We can easily check that the above function satisfies \eqref{cons} (?). 
%In fact,

%

Recall that solutions to the initial boundary value problem for the
equation $u_t-\Delta u=\su$ can blow up in finite time when $\su$ is superlinear, i.e.,
\begin{equation*}
\lim_{u\rightarrow\infty}\frac{\su}{u}=\infty,\ \ \int^{\infty}\frac{1}{\su}du<\infty.
\end{equation*}
See, for example, \cite{CF}. It would be interesting to know if we can allow $\su$ to be bounded above by a linear function. 

The difficult features in our problem are the possible oscillation of $\su$ and the exponential growth conditions we impose on $\frac{1}{\sigma}, \sigma^\prime$. They prevent us from employing the traditional approach of going from lower regularity to higher one. Instead, we will prove \eqref{nunv1} directly. This is done by
%To handle problems due to such conditions, we have managed to 
obtaining a uniform upper bound for $e^{\ep u} $ for $\ep$ sufficiently small. The idea is motivated by a recent paper of the author \cite{X5}. Then we bound $\nabla u$ by $\nvp$, and vice versa, thereby establishing an inequality in $\|\nvp\|_{\infty,\ot}$. This enables us to prove existence for $T$ suitably small. Then we further show that we can extend our solution in the time direction as far away as we want. 
%The difficulty we are facing here is that the exponential integrability $u$ becomes an issue.
%it is really remarkable that we can obtain the boundedness of $u$ even when $\su$ can grow exponentially. Obviously, the coupling between $u$ and $\varphi$ give rise to a strong cancellation effect between $\su$ and $|\nvp|^2$.
%To see the gap between the two sets of assumptions, we take the function
%\begin{equation}
%\sigma(s)=(a_0+s)^\alpha,
%\end{equation}
%where $a_0>0$ and $\alpha$ is a positive number.

This work is organized as follows. Section 2 is largely preparatory. We collect some relevant known results. The proof of the main theorem is contained in Section 3.  

We follow the well-established notation convention whenever possible. Therefore, throughout this paper, the letter $c$ will be used to denote a positive number that depends only on the given data unless stated otherwise. The dot product of two column vectors $\mathbf{F}, \mathbf{G}$ is denoted by $\mathbf{F}\cdot\mathbf{G}$. When we apply the Sobolev embedding theorem, we only deal with the case $N>2$. The case $N=2$ can be handled similarly.

\section{Preliminaries}
In this section we collect some known results for later use. We begin with Gr\"{o}nwall's inequality.
%The first part deals
%with differentiation formulas. In these formulas capital letters represent matrix-valued functions, bold face letters are vector-valued functions, and lower case letters are scalar functions.
%We invoke the following notation conventions
\begin{lemma} Suppose that a differentiable function $h(t)$ satisfies the inequality
	\begin{equation*}
	h^\prime(t)\leq ch(t)+ g(t)\ \ \mbox{on $[0,\infty)$},
	\end{equation*}
	where $c$ is a constant and $g(t)$  a locally integrable function.
	Then
	\begin{equation*}
	h(t)\leq h(0)e^{ct}+\int_{0}^{t}g(\tau)e^{c(t-\tau)}d\tau.
	\end{equation*}
\end{lemma}
%This lemma is the so-called Gr\"{o}nwall's inequality.
We also need the interpolation inequality
\begin{equation}\label{inter}
\|u\|_q\leq \varepsilon\|u\|_r+\varepsilon^{-\mu}\|u\|_\ell,
\end{equation}
where $1\leq \ell\leq q\leq r$ with $\mu=\left(\frac{1}{\ell}-\frac{1}{q}\right)/\left(\frac{1}{q}-\frac{1}{r}\right)$.

The next two lemmas deals with sequences of nonnegative numbers
which satisfy certain recursive inequalities.
\begin{lemma}\label{ynb}
	Let $\{y_n\}, n=0,1,2,\cdots$, be a sequence of positive numbers satisfying the recursive inequalities
	\begin{equation*}
	y_{n+1}\leq cb^ny_n^{1+\alpha}\ \ \mbox{for some $b>1, c, \alpha\in (0,\infty)$.}
	\end{equation*}
	If
	\begin{equation*}
	y_0\leq c^{-\frac{1}{\alpha}}b^{-\frac{1}{\alpha^2}},
	\end{equation*}
	then $\lim_{n\rightarrow\infty}y_n=0$.
\end{lemma}
This lemma can be found in (\cite{D}, p.12).

	\begin{lemma}\label{small}
	Let $\alpha,\lambda\in (0,\infty)$ be given and $\{b_k\}$ a sequence of nonnegative numbers with the property
	\begin{equation}
	b_k\leq b_0+\lambda b_{k-1}^{1+\alpha}\ \ \textup{for $k=1,2,\cdots.$}
	\end{equation}
	If $2\lambda(2b_0)^\alpha<1$, then
	\begin{equation}
	b_k\leq \frac{b_0}{1-\lambda(2b_0)^\alpha}\ \ \textup{for all $k\geq 0$.}
	\end{equation}
\end{lemma}
This lemma can easily be established via induction.

%Obviously, functions that satisfy \eqref{cons} do not have to be bounded above or obey \eqref{cons}. To see an example, just take $\sigma(s)=e^{ds}$.
%, but neither does it satisfy (C2) nor is it . However, the connection between \eqref{cons} and (C3) is not clear. 
%A result in \cite{X3} asserts that \eqref{cons1} is also a consequence of (C3).
%It is natural to ask that if $\sigma(s)$ is differentiable and satisfies (C3) is this enough to guarantee \eqref{cons}?

\section{Proof of the Main result}
The proof of the main theorem is divided into several lemmas. We assume that \eqref{me1}-\eqref{me4} has a weak solution $(u,\vp)$ with $u\in L^\infty(\ot)$. By our discussion in the introduction, this actually implies \eqref{nunv1} and more. We will indicate how we obtain such an (approximate) solution via the Leray-Schauder fixed point theorem near the end of the section. We shall begin with the exponential integrability of $u$ \cite{X10}.
%In this section we first establish the exponential integrability of $u$. This leads to a logarithmic upper bound for $\|u\|_{\infty,\ot}$, 
%The first two lemmas of this section are the two-dimensional versions of the result in \cite{X5}, which have left unproven there. However, they become useful to us only after we have established the third lemma. 
%which puts us in a position to apply a result in \cite{X7} to prove the main theorem. 

\begin{lemma}\label{expb} For each $m\in (0, \frac{1}{c_1\|\varphi_0\|_{\infty,\Omega}})$ there is a positive number $c$ such that
	\begin{equation}\label{ueb}
	\sup_{0\leq t\leq T}\io e^{mu}dx+\ioT \left(e^{mu}|\nabla u|^2+\su e^{mu}|\nvp|^2\right)dxdt\leq c.
	\end{equation}
\end{lemma}	
\begin{proof}
	The weak maximum principle asserts that
	\begin{equation*}
	\|\varphi\|_{\infty,\Omega}\leq \|\varphi_0\|_{\infty,\Omega}.
	\end{equation*}
	We use $\varphi-\varphi_0$ as a test function in \eqref{me2} to obtain
	\begin{eqnarray}
	\io \su|\nvp|^2dx&\leq& \io\su|\nabla\varphi_0|^2dx\leq c.
	%\nonumber\\
	%&\leq &c\io udx+c.
	%\nonumber\\
	%	&\leq &c\left(\io u^2\right)^{\frac{1}{2}}+c.
	\label{pb3}
	\end{eqnarray}
	On the other hand, use $u-u_0$ as a test function in \eqref{me1} to derive
	\begin{eqnarray}
\lefteqn{	\frac{1}{2}\frac{d}{dt}\io(u-u_0)^2dx+\io|\nabla(u-u_0)|^2dx}\nonumber\\
&=&\io\su|\nvp|^2dx+\io(-\pt u_0+\Delta u_0)(u-u_0)dx\nonumber\\
&	=&-\io\su\vp\nvp\nabla(u-u_0)dx-\io\nabla u_0\nabla(u-u_0)dx-\io\pt u_0(u-u_0)dx\nonumber\\
	&\leq &\frac{1}{2}\io|\nabla(u-u_0)|^2dx+\frac{1}{2}\io(u-u_0)^2dx+c+c\io\left((\pt u_0)^2+|\nabla u_0|^2\right)dx.
	\end{eqnarray}
	Use Gr\"{o}nwall's inequality to yield
	\begin{equation}\label{fix2}
	\sup_{0\leq t\leq T}\io u^2dx+\ioT|\nabla u|^2dxd\tau\leq ce^T+cT+c.
	\end{equation}
	Fix
	\begin{equation*}
	K\geq \|u_0\|_{\infty,\ot}.
	\end{equation*}
	For any $C^1$ function $f$ on $\mathbb{R}$ with
	\begin{equation*}
	f>0\ \ \mbox{and}\ \ f^\prime>0
	\end{equation*}
	we use $(f(u)-f(K))^+$ as a test function in \eqref{me1} to obtain
	\begin{equation*}
	\frac{d}{dt}\io\int_{0}^{u}(f(s)-f(K))^+dsdx+\int_{\{u\geq K\}}\left(f^\prime(u)|\nabla u|^2+\su\varphi\nvp f^\prime(u)\nabla u\right)dx=0.
	\end{equation*}
	On the other hand, use $(f(u)-f(K))^+\varphi$ as a test function in \eqref{me1} to yield
	\begin{eqnarray*}
		\int_{\{u\geq K\}}\left(f(u)\su|\nvp|^2+\su\varphi\nvp f^\prime(u)\nabla u\right)dx&=&f(K)\int_{\{u\geq K\}}\su|\nvp|^2dx\nonumber\\
		&\leq& cf(K).
	\end{eqnarray*}
	Combing the preceding two equations, we arrive at
	\begin{eqnarray}
	\lefteqn{	\frac{d}{dt}\io\int_{0}^{u}(f(s)-f(K))^+dsdx+\varepsilon\int_{\{u\geq K\}}\left(f^\prime(u)|\nabla u|^2+f(u)\su|\nvp|^2\right)dx}\nonumber\\
	&&+\int_{\{u\geq K\}}\left((1-\varepsilon)f^\prime(u)|\nabla u|^2+2\su\varphi\nvp f^\prime(u)\nabla u+(1-\varepsilon)f(u)\su|\nvp|^2\right)dx\nonumber\\
	&\leq& cf(K),\label{cons2}
	\end{eqnarray}
	where $\varepsilon\in (0,1)$.
	The last integrand in the above inequality is non-negative if $f$ is so chosen that
	\begin{equation}\label{con}
	\frac{f^\prime(u)}{f(u)}\leq \frac{(1-\varepsilon)^2}{c_1\|\varphi_0\|_{\infty,\Omega}^2}\leq\frac{(1-\varepsilon)^2}{\su\varphi^2}.
	\end{equation}
	We take 
	\begin{equation*}
	f(s)=e^{mu}.
	\end{equation*}
	For \eqref{con}to hold for $\varepsilon$ sufficiently small, it is enough to take
	\begin{equation*}
	m<\frac{1}{c_1\|\varphi_0\|_{\infty,\Omega}^2}.
	\end{equation*}
	Use this in \eqref{cons2}, integrate,  and keep in mind \eqref{fix2}  to derive
	the desired result. The proof is complete.
\end{proof}
We would like to remark that if
\begin{equation}\label{cont}
T\leq 1,
\end{equation}
then the constant $c$ in \eqref{ueb} can be made independent of $T$. This can be easily seen from \eqref{fix2}. For this purpose only, we will assume \eqref{cont} from here on. 

Now let 
\begin{equation*}
w=e^{\varepsilon u},\ \ \ep\in(0,1).
\end{equation*}
% Let $\phi$ be given as in the lemma.
Then $w$ satisfies the problem
\begin{eqnarray}
 w_t-\Delta  w&=&\varepsilon \su|\nvp|^2w\ \ \mbox{in $\ot$}, \label{efp1}\\
 w&=& e^{\varepsilon u_0}\ \ \mbox{on $\partial_p\ot$}. \label{efp11}
\end{eqnarray}
\begin{lemma}\label{pib} Let $w$ be given as above. For each $\frac{N+2}{N}>\ell>1$ and $	0<\varepsilon <\min\left\{1, \frac{1}{2c_1\ell\|\varphi_0\|_{\infty, \Omega}^2} \right\} $ there is a positive number $c$ such that
	\begin{equation}\label{thm2}
	\|w\|_{\infty,\ot}=	\|e^{\ep u}\|_{\infty,\ot}\leq cT^{\frac{1}{2\ell}}\|\nvp\|_{\frac{2\ell}{\ell-1},\ot}^{\frac{N+2}{N+2-N\ell}}+c.
	\end{equation}
\end{lemma}
\begin{proof} 
	Let
	\begin{equation}\label{mic}
	\frac{k}{2}\geq \max\{1, \|e^{ u_0}\|_{\infty, \ot}\}
	\end{equation}
	be selected as below. Set
	\begin{equation*}
	k_n=k-\frac{k}{2^{n+1}},\ \ n=0,1,\cdots
	\end{equation*}
	Then we have
	\begin{eqnarray}
	( w -k_{n})^+\mid_{\partial_p\ot}&=&0.
	\end{eqnarray}
	%	where $\partial_p\ot$ is the parabolic boundary of $\ot$ and^{1+\frac{d}{\varepsilon }}\max_{\partial_p\ot}e^{\varepsilon u_0}
	Use $( w -k_{n+1})^+$ as a test function in \eqref{efp1} to obtain
	\begin{eqnarray}
	\lefteqn{\frac{1}{2}\frac{d}{dt}\io\left[( w -k_{n+1})^+\right]^2dx+\io|\nabla( w -k_{n+1})^+|^2dx}\nonumber\\
	&=&  \io \varepsilon \su|\nvp|^2 w ( w -k_{n+1})^+dx\leq c\io \nvp|^2 w ( w -k_{n+1})^+dx.
	%\nonumber\\
%	&=&\io \varepsilon \su|\nvp|^2 \left[ ( w -k_{n+1})^+\right]^2dx+k_{n+1}\io \varepsilon \su|\nvp|^2  ( w -k_{n+1})^+dx.
%	&\leq &c \|g\|_{\frac{\ell}{\ell-1},\Omega}\left(\io\left( w ( w -k_{n+1})^+\right)^\ell dx\right)^{\frac{1}{\ell}},\ \ \ell>1.
\label{mee1}
	\end{eqnarray}
	%The last step is due to \eqref{cons1}.\tau
	Integrate to get
	\begin{eqnarray}
	\lefteqn{\max_{0\leq t\leq T}\io\left[( w -k_{n+1})^+\right]^2dx+\ioT|\nabla( w -k_{n+1})^+|^2dxdt}\nonumber\\
	&\leq&
	c \|\nvp\|_{\frac{2\ell}{\ell-1},\ot}^2\left(\ioT\left( w ( w -k_{n+1})^+\right)^\ell dxdt\right)^{\frac{1}{\ell}},
	\end{eqnarray}
	where
%	Now we further require
	\begin{equation*}
	1<\ell<\frac{2}{N}+1.
	\end{equation*}
	Let
	\begin{equation*}
	y_n=\left(\ioT\left[( w -k_{n})^+\right]^{2\ell}dxdt\right)^{\frac{1}{\ell}}.
	\end{equation*}
Assume $N>2$. We estimate from the Sobolev embedding theorem that
	\begin{eqnarray}
		%\lefteqn{\ioT\left[(a-k_{n+1})^+\right]^{2s}dxdt}\nonumber\\
	\lefteqn{\int_{0}^{T}\io\left[( w -k_{n+1})^+\right]^{\frac{4}{N}+2}dxdt}\nonumber\\
	&\leq &\int_{0}^{T}\left(\io\left[( w -k_{n+1})^+\right]^{2}dx\right)^{\frac{2}{N}}\left(\io\left[( w -k_{n+1})^+\right]^{\frac{2N}{N-2}}dx\right)^{\frac{N-2}{N}}dt\nonumber\\
		&\leq &c\left(\max_{0\leq t\leq T}\io\left[( w -k_{n+1})^+\right]^2dx\right)^{\frac{2}{N}}\int_{0}^{T}\io\left[\nabla( w -k_{n+1})^+\right]^{2}dxdt\nonumber\\
		&\leq &c\|\su|\nvp|^2\|_{\frac{\ell}{\ell-1},\ot}^{1+\frac{2}{N}}\left(\ioT\left( w ( w -k_{n+1})^+\right)^\ell dxdt\right)^{\frac{N+2}{N\ell}}.
		%\left(\max_{0\leq t\leq T}\io\left[( w -k_{n+1})^+\right]^2dx\right)^{\frac{\ell}{2}}\left(\ioT\left[\nabla( w -k_{n+1})^+\right]^{2}dxdt\right)^{\frac{\ell}{2}}|A_{n+1}|^{1-\frac{\ell}{2}},
	\end{eqnarray}
Set
	\begin{equation*}
	A_{n+1}=\{ w \geq k_{n+1}\}.
	\end{equation*}
	This combined with \eqref{mee1} gives
	\begin{eqnarray}
	%\lefteqn{\ioT\left[( w -k_{n+1})^+\right]^{2\ell}dxdt}\nonumber\\
	y_{n+1}&= &\left(\ioT\left[( w -k_{n+1})^+\right]^{2\ell}dxdt\right)^{\frac{1}{\ell}}\nonumber\\
	&\leq &\left(\ioT\left[( w -k_{n+1})^+\right]^{2\frac{N+2}{N}}dxdt\right)^{\frac{N}{N+2}}|A_{n+1}|^{\frac{1}{\ell}-\frac{N}{N+2}}\nonumber\\
	&\leq &c\|\su|\nvp|^2\|_{\frac{\ell}{\ell-1},\ot}\left(\ioT\left( w ( w -k_{n+1})^+\right)^\ell dxdt\right)^{\frac{1}{\ell}}|A_{n+1}|^{\frac{1}{\ell}-\frac{N}{N+2}}\label{mee2}
	\end{eqnarray}
	On the other hand, we have
	%We can estimate that
	\begin{eqnarray}
	y_n&\geq&\left(\int_{A_{n+1}}\left[( w -k_{n})^+\right]^{2\ell}dxdt\right)^{\frac{1}{\ell}}\nonumber\\
	&=&\left(\int_{A_{n+1}} w ^{\ell}\left[( w -k_{n})^+\right]^{\ell}\left(1-\frac{k_n}{ w }\right)^{\ell}dxdt\right)^{\frac{1}{\ell}}\nonumber\\
	&\geq &\left(\int_{A_{n+1}} w ^{\ell}\left[( w -k_{n})^+\right]^{\ell}\left(1-\frac{k_n}{k_{n+1}}\right)^{\ell}dxdt\right)^{\frac{1}{\ell}}\nonumber\\
	&\geq&\frac{1}{2^{(n+2)}}\left(\int_{A_{n+1}} w ^{\ell}\left[( w -k_{n})^+\right]^{\ell}dxdt\right)^{\frac{1}{\ell}}.\label{mee3}
	\end{eqnarray}
	%The last step is due to our assumption that $k\geq 1$. 
	Furthermore,
	\begin{equation*}
	y_n\geq (k_{n+1}-k_n)^{2}|A_{n+1}|^{\frac{1}{\ell}}=\frac{k^{2}}{2^{2(n+2)}}|A_{n+1}|^{\frac{1}{\ell}}.
	\end{equation*}
	%In view of \eqref{mee2} and \eqref{mee3}, we deduce
	Finally, we arrive at
	\begin{eqnarray}
	y_{n+1}
	%&\leq&\left(\ioT\left[(a-k_{n+1})^+\right]^{2s}dxdt\right)^{\frac{\ell}{s}}|A_{n+1}|^{1-\frac{\ell}{s}}\nonumber\\
	%&\leq&c\|g\|_{\frac{\ell}{\ell-1},\ot}^{2\ell}\ioT\left(a^{}(a-k_{n+1})^+\right)^\ell dxdt|A_{n+1}|^{\frac{(2-s)\ell}{2s}+1-\frac{\ell}{s}}\nonumber\\
	&\leq &c2^{(n+2)}\|\su|\nvp|^2\|_{\frac{\ell}{\ell-1},\ot}y_n|A_{n+1}|^{\frac{1}{\ell}-\frac{N}{N+2}}\nonumber\\
	&\leq &\frac{cb^n}{k^{2\frac{N+2-N\ell}{N+2}}}\|\su|\nvp|^2\|_{\frac{\ell}{\ell-1},\ot}y_n^{1+\frac{N+2-N\ell}{N+2}},
	\end{eqnarray}
	where $b>1$.
	Thus by Lemma \ref{ynb},  if we take $k$ so that
	\begin{equation*}
	y_0\leq c\left(\frac{k^{2\frac{N+2-N\ell}{N+2}}}{\|\su|\nvp|^2\|_{\frac{\ell}{\ell-1},\ot}}\right)^{\frac{N+2}{N+2-N\ell}},
	\end{equation*}
	then
	\begin{equation*}
	 w \leq k.
	\end{equation*}
		Taking into account of  \eqref{mic}, it is enough for us to take
	%This together with implies
	\begin{equation}\label{conk}
	 k= cy_0^{\frac{1}{2}}\|\su|\nvp|^2\|_{\frac{\ell}{\ell-1},\ot}^{\frac{N+2}{2(N+2-N\ell)}}+2 e^{ \|u_0\|_{\infty,\ot}}
	\end{equation}
	%Since $m_i\in \varepsilon ^\infty(0,T; W^{1,2}_0)$, we infer from Theorem 7.15 in (\cite{GT}, p. 162) that there is a positive number $c_0$ such that
	%	\begin{equation}
	%	\io e^{c_0 m_i}dx< \infty.
	%	\end{equation}
	Choose $\varepsilon$ suitably small so that
	\begin{equation*}
	2\ell\varepsilon < \frac{1}{c_1\|\varphi_0\|_{\infty, \Omega}^2}.
	\end{equation*}
	%where $\alpha$ is given as in 
	By Lemma \ref{expb}, we have
	%	Subsequently,
	\begin{equation*}
	y_0\leq\left( \ioT e^{2\ell\varepsilon u}dxdt\right)^{\frac{1}{\ell}}\leq cT^{\frac{1}{\ell}}.
	\end{equation*}
	Plug this into \eqref{conk} to get the desired result.
\end{proof}
%Even though this lemma is similar to the estimate in \cite{X5}, the result there is not suitable for direct application here.

\begin{lemma}\label{nunv} For each $ r \in (N,N+1)$ there is a positive number $c$ such that
	\begin{equation*}
	\|\nabla u\|_{\infty,\ot}\leq cT^{\frac{ r }{2}-\frac{N}{2}}\|\nvp\|^2_{\infty,\ot}+c.
	\end{equation*}
\end{lemma}
%We are ready to prove the main theorem.
\begin{proof}
	%[Proof of the Main Theorem]
Consider the function
	\begin{equation*}
	G=\frac{1}{ (4\pi)^{\frac{N}{2}}}\int_{0}^{t}\frac{1}{(t-\tau)^{\frac{N}{2}}}\int_{\mathbb{R}^N}\exp\left(-\frac{|x-y|^2}{4(t-\tau)}\right)\su|\nvp|^2\chi_{\Omega}dyd\tau.
	\end{equation*}
	We see from (\cite{LSU}, Chapter IV) that $G$ satisfies
	\begin{eqnarray*}
		G_t-\Delta G&=& \su|\nvp|^2\chi_{\Omega}\ \ \ \mbox{in $\mathbb{R}^N\times(0,T)$,}\\
		G(x,0)&=&0\ \ \mbox{on $\mathbb{R}^N$}.
	\end{eqnarray*}
	Furthermore, for each $p> 1$ there is a positive number $c$ such that
	\begin{equation}\label{lsu1}
	\|G_t\|_{p,\ot}+ \|G\|_{L^p(0,T;W^{2,p}(\Omega))}\leq c\|\su|\nvp|^2\|_{p,\ot}.
	\end{equation}
	%We infer from Lemma \ref{expb} that for each $s<2$ and $i\geq 1$ there is a positive number such that
	%\begin{equation}
	%\sup_{\ot}\io\frac{\m^i}{|x-y|^s}dy\leq c.
	%\end{equation}
	Set
	\begin{equation*}
	s=\frac{|x-y|}{2\sqrt{t-\tau}}.
	\end{equation*}
	%For each $ r \in (2,3)$ 
	Let $r$ be given as in the lemma. Then we have 
	\begin{equation*}
	s^r\exp\left(-s^2\right)\leq c(r)\ \ \mbox{on $[0,\infty)$}.
	\end{equation*}
	With this in mind, we estimate
	\begin{eqnarray}
	|\nabla G|&=&\left|\frac{1}{ 2(4\pi)^{\frac{N}{2}}}\int_{0}^{t}\frac{1}{(t-\tau)^{\frac{N}{2}}}\int_{\mathbb{R}^N}(x-y)\exp\left(-s^2\right) \su|\nvp|^2\chi_{\Omega}dyd\tau\right|\nonumber\\
	&\leq &c\int_{0}^{t}\frac{1}{(t-\tau)^{1+\frac{N}{2}}}\int_{\mathbb{R}^N}\frac{\left(2\sqrt{t-\tau}\right)^ r }{|x-y|^{ r -1}}s^ r \exp\left(-s^2\right)| \su|\nvp|^2\chi_{\Omega}dyd\tau\nonumber\\
	&\leq &c\||\nabla \varphi\|^2_{\infty, \ot}\int_{0}^{t}\frac{1}{(t-\tau)^{1+\frac{N}{2}-\frac{ r }{2}}}\int_{\mathbb{R}^N}\frac{\chi_{\Omega}}{|x-y|^{ r -1}}dyd\tau\nonumber\\
	%&&+c\int_{0}^{t}\frac{1}{(t-\tau)^{2-\frac{ r }{2}}}\int_{\mathbb{R}^2}|\m|^{2\gamma-1}\chi_{\ot}\frac{1}{|x-y|^{ r -1}}dyd\tau\nonumber\\
	&\leq &ct^{\frac{ r }{2}-\frac{N}{2}}\|\nabla \varphi\|^2_{\infty, \ot}.\label{nmb2}
	\end{eqnarray}
	Obviously, $F\equiv u-G$ satisfies the problem
	\begin{eqnarray}
	F_t-\Delta F&=&0\ \ \mbox{in $\ot$},\\
	F&=&u_0-G\ \ \mbox{on $\partial_p\ot$.}
	\end{eqnarray}
	%This puts us in a position to invoke . Indeed,  
	%, from which it follows that$(G)_t,\in L^s(\ot)$ .
	%This is enough to ensure that
	We can easily conclude from \eqref{lsu1} and the classical regularity theory for the heat equation (\cite{LSU}, Chapter IV) that
	$\|\nabla F\|_{\infty, \ot}\leq c\|\nabla u_0-\nabla G|\|_{\infty, \ot}\leq cT^{\frac{ r }{2}-\frac{N}{2}}\|\nvp\|^2_{\infty, \ot}+c$.
	Hence we have
	\begin{equation*}
	\|\nabla u\|_{\infty, \ot}\leq cT^{\frac{ r }{2}-\frac{N}{2}}\|\nvp\|^2_{\infty, \ot}+c.
	\end{equation*}
\end{proof}

By (H1), we have
\begin{equation}
\left|\frac{\sigma^\prime(u)}{\su}\right|\leq ce^{(\beta+\gamma)u}.
\end{equation}
Use $\vp-vp_0$ as a test function in \eqref{efvp} to derive
\begin{equation}\label{vpt}
\io|\nvp|^2dx\leq \io\left( e^{(\beta+\gamma)u}|\nabla u|\right)^2dx+c.
\end{equation}
On account of the classical Calder\'{o}n-Zygmund estimate, for each $p>$ there is a positive number $c$ such that
\begin{eqnarray}
\|\varphi\|_{W^{2,p}(\Omega)}&\leq& c\left\| e^{(\beta+\gamma)u}\nabla u\cdot\nvp\right\|_{p,\Omega}+c\|\varphi_0\|_{W^{2,p}(\Omega)}.
\end{eqnarray}
Take $p>N$. Then we derive from the Sobolev embedding theorem and \eqref{inter} that
\begin{eqnarray}
\|\nvp\|_{\infty,\Omega}&\leq &c\|\varphi\|_{W^{2,p}(\Omega)}\nonumber\\
&\leq &c\left\| e^{(\beta+\gamma)u}\nabla u\right\|_{\infty,\Omega}\left\|\nvp\right\|_{p,\Omega}+c\nonumber\\
&\leq &c\left\| e^{(\beta+\gamma)u}\nabla u\right\|_{\infty,\Omega}\left(\varepsilon\left\|\nvp\right\|_{\infty,\Omega}+\frac{1}{\varepsilon^{\frac{p-2}{2}}}\left\|\nvp\right\|_{2,\Omega}\right)+c\nonumber\\
&=&\frac{1}{2}\left\|\nvp\right\|_{\infty,\Omega}+c\left\| e^{(\beta+\gamma)u}\nabla u\right\|_{\infty,\Omega}^{\frac{p-2}{2}}\left(\left\| e^{(\beta+\gamma)u}\nabla u\right\|_{2,\Omega}+c\right)+c.\label{kest}
\end{eqnarray}
Consequently,
\begin{equation}
\|\nvp\|_{\infty,\Omega}\leq c\left\| e^{(\beta+\gamma)u}\nabla u\right\|_{\infty,\Omega}^{\frac{p}{2}}+c.
\end{equation}
By Lemmas \ref{pib} and \ref{nunv},
\begin{eqnarray}
\|\nvp\|_{\infty,\ot}&\leq& c\left\| e^{(\beta+\gamma)u}\nabla u\right\|_{\infty,\ot}^{\frac{p}{2}}+c\nonumber\\
&\leq &\left(cT^{\frac{1}{2\ell}}\|\nvp\|_{\frac{2\ell}{\ell-1},\ot}^{\frac{N+2}{N+2-N\ell}}+c\right)^{\frac{p(\beta+\gamma)}{2\ep}}\left(cT^{\frac{ r }{2}-\frac{N}{2}}\|\nvp\|^2_{\infty,\ot}+c\right)^{\frac{p}{2}}+c\nonumber\\
&\leq &cT^a\|\nvp\|^b_{\infty,\ot}+c,\label{r11}
\end{eqnarray}
where $a, b $ are two positive numbers. Obviously, we can take
\begin{equation}
b=\frac{(N+2)(\beta+\gamma)p}{2\ep(N+2-N\ell)}+p>1.
\end{equation}
In view of \eqref{cont}, $a$ is the smallest power of $T$ that appear in the product in \eqref{r11}.
We will first show $\vp\in L^\infty(\ot) $ for $T$ suitably small. Then extend the solution in the time direction. To this end,
remember that $c$ in \eqref{r11} is independent of $T$. 
Set \begin{equation*}
\ep=cT^{a}.
\end{equation*}
Consider the function $g(\tau)=\ep \tau^{b}-\tau+c$ on $[0, \infty)$. Then \eqref{r11} implies 
\begin{equation}\label{key}
g\left(\|\nabla \vp\|_{\infty,\Omega\times [0,s] }\right)\geq 0\ \ \mbox{for each $s\in[0,T]$}.
\end{equation}The function $g$ achieves its minimum value at $\tau_0=\frac{1}{\left(\ep b\right)^{\frac{1}{b-1}}}$. The minimum value
\begin{eqnarray*}
	g(\tau_0)&=&\frac{\ep}{\left(\ep b\right)^{\frac{b}{b-1}}}-\frac{1}{\left(\ep b\right)^{\frac{1}{b-1}}}+c\nonumber\\
	&=&c-\frac{\ep(b-1)}{\left(\ep b\right)^{\frac{b}{b-1}}}\leq -\ep,
\end{eqnarray*}
provided that
\begin{equation}\label{cont1}
(c+\ep)\ep^{\frac{1}{b-1}} \leq\frac{b-1}{b^{\frac{b}{b-1}}}.
\end{equation}  In addition to this, we require that
\begin{equation}\label{cont2}
\|\nabla \vp(\cdot,0)\|_{\infty,\Omega}\leq \tau_0.
\end{equation}  
%Recall that our solution $(u, \vp)$ here is the local solution constructed in \cite{X6}.
 If $|\nvp|$ is bounded, then $\nabla u$ is H\"{o}lder continuous. This can be inferred from differentiating \eqref{me1} with respect to $x_i, i=1,\cdots, N$, respectively. We claim that $\|\nabla \vp\|_{\infty,\Omega\times[0,t]}$ is also a continuous function of $t$. To see this, fix a sequence $\{t_n\}\subset [0,T]$ with $\lim_{n\rightarrow\infty}t_n=t_0$. Define
 \begin{eqnarray}
 \vp_n&=&\vp(x,t_n),\\
 \theta(x,t)&=&-\frac{\sigma^\prime(u(x,t))}{\sigma(u(x,t))}\nabla u(x,t).
 \end{eqnarray}Then we have
 \begin{eqnarray}
 \Delta\vp_n&=&\theta(x,t_n)\cdot\nabla\vp_n\ \ \mbox{in $\Omega$,}\label{r1}\\
 \vp_n&=&\vp_0(x, t_n)\ \ \mbox{on $\po$}.\label{r2}
 \end{eqnarray}
 By a calculation similar to \eqref{kest}, we obtain that $\{\vp_n\}$ is precompact in $C^1(\overline{\Omega})$. We can extract a subsequence of $\{\vp_n\}$, still denoted by $\{\vp_n\}$, such that
 \begin{equation*}
 \vp_n\rightarrow \vp^*\ \ \mbox{strongly in $C^1(\overline{\Omega})$.}
 \end{equation*} 
 Pass to the limit in \eqref{r1}-\eqref{r2} to get
  \begin{eqnarray}
 \Delta\vp^*&=&\theta(x,t_0)\cdot\nabla\vp^*\ \ \mbox{in $\Omega$,}\label{r3}\\
 \vp^*&=&\vp_0(x,t_0)\ \ \mbox{on $\po$}.
 \end{eqnarray}
 By the uniqueness of a solution to the above problem, we have
 \begin{equation}
 \vp^*=\vp(x, t_0).
 \end{equation}
 Consequently, the whole sequence $\{\vp_n\}$ converges to $\vp(x, t_0)$ in $C^1(\overline{\Omega})$.
 Now we can conclude from \eqref{key} that $\|\nabla \vp\|_{\infty,\ot}\leq \tau_0$ whenever \eqref{cont}, \eqref{cont1}, and \eqref{cont2} all hold. Condition \eqref{cont1} can be achieved easily by taking $T$ suitably small. As for \eqref{cont2}, notice that $\vp(x,0)$ satisfies the boundary value problem
 \begin{eqnarray}
 \Delta \vp(x,0)&=& -\frac{\sigma^\prime(u_0(x,0))}{\sigma(u_0(x,0))}\nabla u_0(x,0)\cdot\nvp(x,0)\ \ \mbox{in $\Omega$,}\label{vpo}\\
 \vp(x,0)&=&\vp_0(x,0)\ \ \mbox{on $\po$.}
 \end{eqnarray}
 % \eqref{po1}-\eqref{po2}. 
 Our assumptions on $u_0, \vp_0, \sigma(s)$ imply that the coefficient on the right-hand side of \eqref{vpo} are bounded. We can conclude from \eqref{kest} 
% the Calder\'{o}n-Zygmund estimate and the interpolation inequality
  that
 %are strong enough to guarantee that 
 $|\nabla \vp(x,0)|$ is bounded. We can also obtain \eqref{cont2} for suitably small $T$. In summary, if  $T_0$ is the largest $T$ such that \eqref{cont}, \eqref{cont1}, and \eqref{cont2} all hold, then
\begin{equation}\label{rr1}
\|\nabla \vp\|_{\infty,\Omega\times[0,T_0]}\leq \tau_0.
\end{equation}
We consider $(u(x,t+T_0), \vp(x,t+T_0 ))$ on $\Omega\times[0,T_0]$. Conditions \eqref{cont1} and \eqref{cont2} still hold, and so does \eqref{rr1}.  Therefore, 
we can extend the solution in the time direction as far away as we want.
In a sense, \eqref{r11} is a stationary version of Lemma \ref{small} (also see \cite{SMG,X7,X9}).

Existence of a solution can be established via the Leray-Schauder theorem (\cite{GT}, p. 280). To this end, we define an operator $\mathbb{B}$ from $C(\overline{\ot})$ into $C(\overline{\ot})$ as follows: We say $u=\mathbb{B}(v)$ if $v\in C(\overline{\ot})$ and $u$ is the solution of the problem
\begin{eqnarray}
\pt u-\Delta u&=&\sigma(v)|\nabla\vp|^2\ \ \mbox{in $\ot$},\\
u&=&u_0\ \ \mbox{on $\partial_p\ot$,}
\end{eqnarray}
where $\vp$ solves the boundary value problem
\begin{eqnarray}
-\mdiv\left(\sigma(v)\nabla\vp\right)&=&\ \ \mbox{in $\ot$,}\\
\vp&=&\vp_0\ \ \mbox{on $\Sigma_T$.}
\end{eqnarray}
To see that $\mathbb{B}$ is well-defined, we conclude from \eqref{r21} for each $p>1$ there is a positive number $c$ depending on the continuity of $\sigma(v)$ and $\po$ such that
\begin{equation}\label{npu}
\|\nvp\|_{p,\Omega}\leq c\|\nvp_0\|_{p,\Omega}.
\end{equation}
This is more than enough to guarantee that $u$ is H\"{o}lder continuous in $\overline{\ot}$. Since the two problems in the definition of $\mathbb{B}$ are both linear, we can conclude that $\mathbb{B}$ is continuous and maps bounded sets into precompact ones. We still need to show that there is a positive number $c$ such that
\begin{equation}\label{ub}
%^{\alpha,\frac{\alpha}{2}}
\|u\|_{C(\overline{\ot})}\leq c
\end{equation}
for all $u\in \mathbb{B}$ and $\ep\in(0,1)$ satisfying $u=\ep\mathbb{B}(u)$. This equation is equivalent to
\begin{eqnarray}
\pt u-\Delta u&=&\ep\sigma(u)|\nabla\vp|^2\ \ \mbox{in $\ot$},\\
-\mdiv\left(\sigma(u)\nabla\vp\right)&=&\ \ \mbox{in $\ot$,}\\
u&=&\ep u_0\ \ \mbox{on $\partial_p\ot$,}\\
\vp&=&\vp_0\ \ \mbox{on $\Sigma_T$.}
\end{eqnarray}
To obtain \eqref{ub}, we have to apply our early proof to this problem. We only mention that by the calculations in \eqref{nmb2}, \eqref{npu} implies that $|\nabla u|\in L^\infty(\ot)$, and thus \eqref{kest} remains valid. Note that \eqref{npu} is only used to justify the regularity of the solution. In particular, the constant $c$ in \eqref{npu} does not appear elsewhere in our proof. We have all the ingredients necessary to conclude \eqref{ub}. This finishes the proof of the main theorem.

\end{document}